\documentclass[final]{amsart}
\usepackage{amsfonts,amsmath,amssymb,mathtools}
\usepackage{graphicx}
\usepackage[comma, authoryear]{natbib}
\usepackage[colorlinks,citecolor=blue]{hyperref}
\usepackage{pdflscape,pdfsync}
\usepackage{xspace,colortbl,multirow}
\usepackage[dvipsnames]{xcolor}
\usepackage{color}
\usepackage{comment}

\usepackage{ulem}
\usepackage{setspace}
\newtheorem{thm}{Theorem}[section]
\newtheorem{lem}[thm]{Lemma}
\newtheorem{prop}[thm]{Proposition}

\theoremstyle{definition}

\newtheorem{rem}{Remark}[section]

\newtheorem{hyp}{Hypothesis}[section]

\numberwithin{equation}{section}

\def\disp{\displaystyle}

\def\crochet#1{\langle #1 \rangle}

\newcommand{\dA}{\ensuremath{\mathbb{A}}}
\newcommand{\dB}{\ensuremath{\mathbb{B}}}
\newcommand{\dC}{\ensuremath{\mathbb{C}}}

\newcommand{\dF}{\ensuremath{\mathbb{F}}}

\newcommand{\dH}{\ensuremath{\mathbb{H}}}
\newcommand{\I}{\ensuremath{\mathbb{I}}}
\newcommand{\dJ}{\ensuremath{\mathbb{J}}}

\newcommand{\dM}{\ensuremath{\mathbb{M}}}
\newcommand{\dN}{\ensuremath{\mathbb{N}}}

\newcommand{\dR}{\ensuremath{\mathbb{R}}}

\newcommand{\dW}{\ensuremath{\mathbb{W}}}

\newcommand{\cC}{\ensuremath{\mathcal{C}}}

\newcommand{\cE}{\ensuremath{\mathcal{E}}}
\newcommand{\cF}{\ensuremath{\mathcal{F}}}

\newcommand{\cI}{\ensuremath{\mathcal{I}}}
\newcommand{\cJ}{\ensuremath{\mathcal{J}}}

\newcommand{\cL}{\ensuremath{\mathcal{L}}}
\newcommand{\cM}{\ensuremath{\mathcal{M}}}

\newcommand{\cS}{\ensuremath{\mathcal{S}}}

\newcommand{\cW}{\ensuremath{\mathcal{W}}}

\def\i1{ [-\infty,\infty]}

\def\sn{\sqrt{n}\; }

\def\nt{{\lfloor n t \rfloor}}

\def\nt{{\lfloor n t \rfloor}}

\usepackage[color, notref, notcite]{showkeys}
\usepackage{enumerate}
\definecolor{labelkey}{rgb}{0.6,0,1}

\newcommand{\BR}[1]{\textcolor{red}{#1}}

\begin{document}

\title[CLT for martingales-II: convergence in weak dual topology]{Central limit theorems for martingales-II: convergence in the weak dual topology, with corrections}

\author{Bruno R\'{e}millard}
\address{GERAD and Department of Decision Sciences, HEC Montr\'{e}al\\
3000, che\-min de la C\^{o}\-te-Sain\-te-Ca\-the\-ri\-ne,
Montr\'{e}al (Qu\'{e}\-bec), Canada H3T 2A7}
\email{bruno.remillard@hec.ca}

\author{Jean Vaillancourt}
\address{Department of Decision Sciences, HEC Montr\'{e}al\\
3000, che\-min de la C\^{o}\-te-Sain\-te-Ca\-the\-ri\-ne,
Montr\'{e}al (Qu\'{e}\-bec), Canada H3T 2A7} \email{jean.vaillancourt@hec.ca}

\thanks{Partial funding in support of this work was provided by the Natural
Sciences and Engineering Research Council of Canada and  the Fonds
qu\'e\-b\'e\-cois de la re\-cher\-che sur la na\-tu\-re et les
tech\-no\-lo\-gies.
We would like to thank Bouchra R. Nasri for suggesting Proposition \ref{prop:Whitt_composition}.}


\begin{abstract}
A convergence theorem for martingales with c\`adl\`ag trajectories (right continuous with left limits everywhere) is obtained in the sense of the weak dual topology on Hilbert space, under conditions that are much weaker than those required for any of the usual Skorohod topologies.
Examples are provided to show that these conditions are also very easy to check and yield
useful asymptotic results, especially when the limit is a mixture of stochastic processes with discontinuities.
\end{abstract}


\maketitle

\section{Introduction}
Let $D= D[0,\infty)$ be the space of real-valued c\`adl\`ag trajectories (right continuous with left limits everywhere).
The space is a natural choice for describing phenomena exhibiting bounded jumps, hence devoid
of essential singularities (and no jump at the origin).  In the literature, there are many examples of sequences of
$D$-valued martingales $M_n$ such that the finite dimensional distributions (f.d.d.) converge to those of $W\circ A$,
where $W$ is Brownian motion and $A$ is a non-decreasing process starting at zero and independent of $W$.
However, if $A$ is not continuous, the f.d.d. convergence cannot be extended in general to any of the three familiar topologies on $D$:
the $\mathcal{M}_1$-topology, the $\mathcal{J}_1$-topology or the $\mathcal{C}$-topology.

The focus of this paper is on those instances where the sequence
of $D$-valued processes of interest does not converge in any of these topologies.
This motivates the consideration of an alternative topology: the weak dual topology
on the dual of a Hilbert space, which coincides with the pointwise convergence in the Hilbert space in question,
namely that associated with locally square-integrable trajectories.
Since weak convergence of probability measures is at the core of this paper,
and in order to avoid the potential confusion brought upon by the overuse of the term ``weak'',
we will from now on refer to this weaker topology as the $\cL^2_w$-topology (and all related terms accordingly).

The paper is organized as follows. Section \ref{sec:ell_two} presents the $\cL^2_w$-topology.
The corresponding CLT for sequences of $D$-valued square integrable martingales is in Section \ref{sec:m1_clt}, along with the required preliminaries about martingales.
Section \ref{sec:ex} provides instances where even
$\mathcal{M}_1$-convergence (let alone $\mathcal{J}_1$- or $\mathcal{C}$-convergence) fails for interesting
sequences of $D$-valued processes, while $\cL^2_w$-convergence does occur.
Examples of applications are worked out in Section \ref{sec:ex}. More precisely, Section \ref{ssec:random_sums}  deals with stock price approximation and random sums, and it is also shown in Remark \ref{rem:error} that many results about the limiting behaviour of renewal counting processes in the literature are erroneous, showing the difficulty to work with $\cJ_1$ or $\cM_1$ convergence when the limit is not continuous. Next, in Section \ref{ssec:occupation_times}, we study the  occupation time for planar Brownian motion; in particular, we complete the proof of Theorem 3.1 of  \citet{Kasahara/Kotani:1979} and we correct their Theorem 3.2.
Ancillary results as well as most definitions and notations are presented in Appendix \ref{app:remvail},
while any unexplained terminology can be found in either \citet{Ethier/Kurtz:1986} or \citet{Jacod/Shiryaev:2003}.
The main proofs are collected in Appendix \ref{app:remvail_prfs}.


\section{Locally square integrable processes and $\cL^2_w$-convergence}\label{sec:ell_two}
Denote by $\cL^2_{loc}$ the space of real-valued Lebesgue measurable functions on $[0,\infty)$ which are
locally square integrable with respect to Lebesgue measure, in the sense that their restriction to $[0,T]$ belongs to
Hilbert space $\cL^2[0,T]$ for every $T>0$.
The (squared) $\cL^2[0,T]$-norm is denoted by $|x|_T^2:=\int_0^Tx^2(t)dt$.

A (deterministic) sequence $x_n\in\cL^2_{loc}$ is said to $\cL^2_w$-converge to $x\in\cL^2_{loc}$ if and only if the sequence of scalar products $\int_0^Tx_n(t)f(t)dt \to \int_0^Tx(t)f(t)dt$ converges as $n\to\infty$, for each $f\in \cL^2[0,T]$ and each $T>0$.
There is a metric on $\cL^2_{loc}$ which is compatible with the $\cL^2_w$-topology,
turning it into a complete separable metric (hereafter Polish) space --- see Appendix \ref{app:remvail}.

\begin{rem}\label{rem:DinsideL2}
Note that $D\subset\cL^2_{loc}$ since c\`adl\`ag functions are bounded on bounded time sets.
Furthermore $\mathcal{M}_1$-convergence of a (deterministic) sequence $x_n\in D$ to a limit $x\in D$
implies $\sup_n\sup_{t\in[0,T]}|x_n(t)-x(t)|<\infty$ for every $T>0$, since $\{x_n\}$ is a $\mathcal{M}_1$-relatively compact set
and hence uniformly bounded on finite time intervals.
By the dominated convergence theorem, convergence of $x_n$ to the same limit $x$ ensues in the
$\cL^2[0,T]$-norm topology for every $T>0$ and hence so does $\cL^2_w$-convergence; indeed,
a (deterministic) sequence $x_n\in\cL^2[0,T]$ converges to $x\in\cL^2[0,T]$ in the (strong) norm topology,
if and only if both $x_n$ $\cL^2_w$-converges to $x$ and $|x_n|_T$ converges to $|x|_T$.
Thus $x\mapsto|x|_T$ is not $\cL^2_w$-continuous even though $T\mapsto|x|_T$ is continuous.
\end{rem}

We next turn to weak convergence of probability measures on the Borel subsets of $\cL^2_{loc}$ for the $\cL^2_w$-topology.
All $\cL^2_{loc}$-valued $\mathbb{F}$-processes are built on the same filtered probability space $(\Omega,\cF,\dF, P)$ and
adapted to a filtration $\dF = (\cF_t)_{t\ge 0}$ satisfying the usual conditions (notably, right continuity of the filtration and completeness).
Since $\cL^2_{loc}$ is a Polish space, families of probability measures are relatively compact
if and only if they are tight, by Prohorov's Theorem. Here is a quick test for $\cL^2_w$-tightness.

\begin{lem}\label{lem:L2tightness}
A sequence of $\cL^2_{loc}$-valued $\mathbb{F}$-processes $X_n$ is $\cL^2_w$-tight if and only if
the sequence of random variables $\bigl\{|X_n|_T; n\ge1\bigr\}$ is tight, for each $T>0$.
\end{lem}

\begin{proof}
Balls in $\cL^2[0,T]$ are relatively compact in the $\cL^2_w$-topology and because it is a Polish space, tightness is equivalent to relative compactness.
\end{proof}

Weak convergence of a sequence of $\cL^2_{loc}$-valued processes $X_n$ to an $\cL^2_{loc}$-valued process $X$
is denoted by $X_n\stackrel{\cL^2_w}{\rightsquigarrow} X$. For convenience, write $\cI(f)(t) = \int_0^t f(s)ds $ for any $t\ge 0$.
This convergence is characterized by the following result proven in Appendix \ref{pf:L2characterization}.

\begin{thm}\label{thm:L2characterization}
$X_n\stackrel{\cL^2_w}{\rightsquigarrow} X$ holds if and only if
\begin{enumerate}
\item[(i)]
$\bigl\{|X_n|_T; n\ge1\bigr\}$ is tight, for each $T>0$;

\item[(ii)]
There is an $\cL^2_{loc}$-valued process $X$ such that $\cI(X_n)  \stackrel{f.d.d.}{\rightsquigarrow} \cI(X)$.

\end{enumerate}
\end{thm}

The Cram{\'e}r-Wold device yields the equivalence of statement Theorem \ref{thm:L2characterization} $(ii)$ to
$\cI(X_n h)  \stackrel{f.d.d.}{\rightsquigarrow} \cI(X h)$, for any $h\in\cL^2_{loc}$.
Note that $X_n\stackrel{\cL^2_w}{\rightsquigarrow} X$ does not necessarily imply $X_n \stackrel{f.d.d.}{\rightsquigarrow} X$;
in fact it may even happen that $X_n(t) \stackrel{Law}{\rightsquigarrow} X(t)$ does not hold for any $t\in[0,T]$, even though
the integrals $\cI(X_n)(t)$ converge and their limits are almost surely differentiable everywhere.
Moreover, if $X_n\in D$, $X_n(t)$ is arbitrarily close to $\dfrac{\cI(X_n)(t+\epsilon)-\cI(X_n)(t)}{\epsilon}$.

\begin{rem}\label{rem:RdExtension1}
Extension to processes with $\dR^d$-valued trajectories, of $\cL^2_{loc}$-valued processes, $\cL^2_w$-convergence, and so on ---
through writing $|X|_t:=\sum_{i=1}^d|X_i|_t$, $X(s)f(s):=\sum_{i=1}^d X_i(s)f_i(s)$, etc.,
when $X=(X_1,\ldots,X_d)$ and $f=(f_1,\ldots,f_d)$ --- leads to immediate confirmation of
Remark \ref{rem:DinsideL2}, Lemma \ref{lem:L2tightness} and Theorem \ref{thm:L2characterization}
for processes with trajectories in $\dR^d$ as well. We need this clarification for the next statement.
\end{rem}

By Remark \ref{rem:RdExtension1}, $|(X_n,Y_n)|_T=|X_n|_T+|Y_n|_T$, and this sequence is tight if both sequences $|X_n|_T$ and $|Y_n|_T$ are tight. Also,
$\cI(X_n,Y_n)= (\cI(X_n),\cI(Y_n) )$, so the next result follows readily from Theorem \ref{thm:L2characterization}.

\begin{lem}\label{lem:lem_pairs}
If $\cL^2_{loc}$-valued processes $X_n$, $Y_n$, $X$ and $Y$ are such that
$X_n\stackrel{\cL^2_w}{\rightsquigarrow}X$, $Y_n \stackrel{\cL^2_w}{\rightsquigarrow}Y$ and
$\left( \cI(X_n) , \cI( Y_n) \right ) \stackrel{f.d.d.}{\rightsquigarrow}
\left( \cI(X)  , \cI(Y) \right )$,
then $(X_n,Y_n) \stackrel{\cL^2_w}{\rightsquigarrow} (X,Y)$.
\end{lem}

\begin{rem}\label{rem:RdExtension2}
Lemma \ref{lem:lem_pairs} also extends at once to processes with $\dR^d$-valued trajectories
and to $d$-tuples of processes instead of just pairs.
We proceed with $d=1$ for the rest of this section, since several of the statements involve the $\mathcal{M}_1$-topology,
which is not compatible with the linear structure on $D$ and therefore cannot be treated coordinatewise, save in some special cases
--- see \citet{Whitt:2002}.
\end{rem}

\begin{rem}\label{rem:L2normtightness}
In the light of Proposition \ref{prop:Cincreasing}, a sequence of $D$-valued processes $X_n$ satisfying both
$X_n \stackrel{f.d.d.}{\rightsquigarrow} X$ and $|X_n| \stackrel{f.d.d.}{\rightsquigarrow} |X|$ for some $X\in D$,
must also satisfy $|X_n|  \stackrel{\mathcal{C}}{\rightsquigarrow} |X|$ since the processes $|X_n|_t, |X|_t$ are nondecreasing in $t\ge 0$; hence
the sequence of continuous processes $\bigl\{|X_n|; n\ge1\bigr\}$ is $\mathcal{C}$-tight.
By Remark \ref{rem:DinsideL2}, $D$-valued processes $X$, $X_n$ for which there holds
$X_n \stackrel{\mathcal{M}_1}{\rightsquigarrow} X$ automatically verify both $X_n\stackrel{\cL^2_w}{\rightsquigarrow} X$ and
$|X_n| \stackrel{f.d.d.}{\rightsquigarrow} |X| $ (as well as $X_n \stackrel{f.d.d.}{\rightsquigarrow} X$, by definition of
$\mathcal{M}_1$-convergence); thus, $|X_n| \stackrel{\mathcal{C}}{\rightsquigarrow} |X|$ follows.
As a result, so does $|X_n-X| \stackrel{\mathcal{C}}{\rightsquigarrow} 0$.
\end{rem}

The following result follows directly from Remark \ref{rem:L2normtightness}.

\begin{prop}\label{prop:norm}
Suppose $X$, $X_n$ are  $D$-valued processes satisfying $X_n \stackrel{\mathcal{M}_1}{\rightsquigarrow} X$.
Then $X_n\stackrel{\cL^2_w}{\rightsquigarrow} X$, $|X_n| \stackrel{\mathcal{C}}{\rightsquigarrow} |X|$
and $|X_n-X|  \stackrel{\mathcal{C}}{\rightsquigarrow} 0$.
\end{prop}

One of the main reasons for choosing the $\cL^2_w$-topology is the following theorem stating sufficient conditions
for the $\cL^2_w$-continuity of the composition mapping for sequences of $D$-valued processes.

\begin{thm}\label{thm:comp}
Suppose $X$, $Y$, $X_n$, $Y_n$ are $D$-valued processes such that
$X_n\stackrel{\cC}{\rightsquigarrow} X$ and $(X_n,Y_n) \stackrel{f.d.d.}{\rightsquigarrow} (X,Y)$ hold,
with $Y_n$, $Y$ non-negative, non-decreasing such that $Y_n(t) \to\infty$ as $t\to\infty$ almost surely
and either of the following hold:
a) $Y_n \stackrel{\cM_1}{\rightsquigarrow} Y$ with $Y_n^{-1}(0)=0$ for every $n$; or
b) $Y_n^{-1} \stackrel{\cM_1}{\rightsquigarrow} Y^{-1}$ with $Y_n(0)=0$ for every $n$.
Then $X_n\circ Y_n \stackrel{\cL^2_w}{\rightsquigarrow} X\circ Y$.
\end{thm}

The proof of the theorem is relegated to Appendix \ref{pf:comp}. Note that even in the $\cM_1$-topology, the composition mapping is not continuous in general, even when $X$ is a Brownian motion; see, e.g., Proposition \ref{prop:Whitt_composition}.


\section{$D$-valued martingales and a CLT in the $\cL^2_w$-topology}\label{sec:m1_clt}

Suppose that $M_n$ is a sequence of $D$-valued square integrable $\mathbb{F}$-martingales started at $M_n(0)=0$.
By square integrable martingale we mean those satisfying $E\{M^2(t)\}<\infty$ for every $t\ge0$.
Note that because of possible discontinuities of trajectories, the quadratic variation $[M_n]$ can be distinct from its predictable compensator
$A_n:=\crochet{ M_n }$. The largest jump is denoted by $J_T(M_n) := \sup_{s\in [0,T]}|M_n(s)-M_n(s-)|$.
Some useful assumptions are formulated next.
\begin{hyp}\label{hyp:an_unbounded}
All of the following hold:
\begin{enumerate}
\item[(a)] $A_n(t) \to\infty$ as $t\to\infty$ almost surely, for each fixed $n\ge1$;
\item[(b)] There is a $D$-valued process $A$ such that \\
(i) $A_n \stackrel{f.d.d.}{\rightsquigarrow} A$; \\
(ii) For all $t\ge 0$, $\disp \lim_{n\to\infty} E\left\{A_n(t)\right\}=E\left\{A(t)\right\} < \infty$; \\
(iii) $A(t) \to\infty$ as $t\to\infty$ almost surely.
\end{enumerate}
\end{hyp}
Writing the inverse process for $A_n$ as $\tau_n(s) = \inf\{t\ge0; A_n(t)>s\}$,
one defines the rescaled $\mathbb{F}_{\tau_n}$-martingale $W_n = M_n\circ \tau_n$,
with compensator $\crochet{ W_n } := A_n\circ \tau_n$.
\begin{hyp}\label{hyp:bm_limit}
\BR{$J_t(W_n)\stackrel{Law}{\rightsquigarrow} 0$ as $n\to\infty$} and
$\disp \lim_{n\to\infty}E\{\langle W_n\rangle_t\}  =t$, for all $t\ge0$.
\end{hyp}
\begin{rem}
By the right-continuity of $A_n$, $A_n\circ \tau_n(t)\ge t$ holds; therefore,
$\lim_{n\to\infty} E\crochet{ W_n }_t  =t$ is equivalent to $\crochet{ W_n }_t \stackrel{L^1}{\rightsquigarrow} t$.
Hypothesis \ref{hyp:bm_limit} yields $\crochet{ W_n }_t \stackrel{Law}{\rightsquigarrow} t$ for any $t\ge 0$;
by Proposition \ref{prop:Cincreasing}, $\crochet{ W_n }$ is $\mathcal{C}$-tight.
\end{rem}
An inhomogeneous L\'evy process $A$ is an $\mathbb{F}$-adapted $D$-valued process with independent increments (with respect to
$\mathbb{F}$), without fixed times of discontinuity and such that $A_0=0\in\dR$. We assume also here that $A$ has deterministic characteristics (in the sense of the L\'evy-Khintchine formula), so all L\'evy processes here are also semimartingales ---
see \citet[Theorem II.4.15]{Jacod/Shiryaev:2003}. The inhomogeneity (in time) means that stationarity of the increments, usually required of L\'evy processes, is lifted. This choice reflects the existence of weak limits exhibiting independent increments without homogeneity in time or space, in some applications. All processes in the present section are built on $D$, ensuring continuity in probability of the trajectories, another usual requirement.

\begin{hyp}\label{hyp:disc_disc1}
There is a $D$-valued process $A$ started at $A(0)=0$, such that
\begin{enumerate}
\item[(a)] $\tau(s) = \inf\{t\ge0; A(t)>s\}$ is an inhomogeneous L\'evy process;
\item[(b)] Either of the following conditions hold: \\
(i) $A_n \stackrel{f.d.d.}{\rightsquigarrow} A$ and for every $n$, $\tau_n(0)=0$; \\
(ii) $\tau_n \stackrel{f.d.d.}{\rightsquigarrow} \tau$ and for every $n$, $A_n(0)=0$.
\end{enumerate}
\end{hyp}

\begin{rem}\label{rem:jumps_tau}
By Proposition \ref{prop:inversecontinuity} and Remark \ref{rem:M1increasing}, Hypothesis \ref{hyp:disc_disc1}.b(i) implies both
$A_n \stackrel{\mathcal{M}_1}{\rightsquigarrow} A$ and the $\mathcal{M}_1$-continuity of mapping $A_n\mapsto\tau_n$,
so $\tau_n \stackrel{\mathcal{M}_1}{\rightsquigarrow} \tau$ as well.
The reverse argument holds under Hypothesis \ref{hyp:disc_disc1}.b(ii); therefore, Hypothesis \ref{hyp:disc_disc1} automatically
implies both $\tau_n \stackrel{\mathcal{M}_1}{\rightsquigarrow} \tau$ and $A_n \stackrel{\mathcal{M}_1}{\rightsquigarrow} A$.
Invoking Remark \ref{rem:M1increasing} again confirms $(A_n,\tau_n) \stackrel{\mathcal{M}_1}{\rightsquigarrow} (A,\tau)$ as well,
since the mapping $A\mapsto(A,\tau)$ is also $\mathcal{M}_1$-continuous, under condition (i); and similarly for mapping
$\tau\mapsto(A,\tau)$, under condition (ii).
\end{rem}

Our main result is the following CLT, proven in Appendix \ref{pf:clt_mart}.

\begin{thm}\label{thm:clt_mart}
Assume that both Hypotheses \ref{hyp:an_unbounded} and \ref{hyp:bm_limit} hold.
Assume that either: a) Hypothesis \ref{hyp:disc_disc1} holds;
or b) $A$ is an inhomogeneous L\'evy process and $\tau_n(0)=0$, for every $n$.
Then there holds $W_n \stackrel{\mathcal{C}}{\rightsquigarrow} W$ and
$(M_n,A_n,W_n) \stackrel{\cL^2_w}{\rightsquigarrow} (M,A,W)$,
with $M = W\circ A$ and $W$ is a Brownian motion, independent of process $A$.
\end{thm}

\section{Examples of application}\label{sec:ex}

\subsection{Price approximation and random sums}\label{ssec:random_sums}

Following \cite{Chavez-Casillas/Elliott/Remillard/Swishchuk:2019,Swishchuk/Remillard/Elliott/Chavez-Casillas:2019,Guo/Remillard/Swishchuk:2020}
who studied limit order books, the price structure on markets can be written as $X_t = \sum_{k=1}^{N_t} V_0(\xi_k)$, where $\xi = (\xi_k)_{k\ge 1}$ is a finite Markov chain independent of a counting process $N$. Here, $N_t$ can be seen as the number of orders executed on a stock up to time $t$. If $\mu$ is the mean of $V(\xi)$ under the stationary measure $\pi$, then $X_t = \sum_{k=1}^{N_t} V(\xi_k)  + \mu N_t$, where $V(x) = V_0(x)-\mu$.
Next, according to \cite[Theorem 7.7.2 and Example 7.7.2]{Durrett:1996}, one can find $f$ bounded so that $ Tf-f = Lf = -V$ and $V(\xi_k) = Y_k+ Z_k-Z_{k+1}$, with $Z_k = Tf(\xi_{k-1})$ and $Y_k = f(\xi_k)-Tf(\xi_{k-1})$ is a $L^2$ ergodic martingale difference. As a result,
$\disp
\sum_{k=1}^n V(\xi_k) = \sum_{k=1}^n Y_k + Z_1-Z_{n+1}$.
Then, if $\sigma^2 = E(Y_k^2)$, it follows that   $\cW_n(t) = \dfrac{1}{\sigma \sn}\sum_{k=1}^ \nt Y_k $ converges in $\cC$ to a Brownian motion $\cW$.
Next, assume that  $A_n(t) = N_{a_nt}/n$ converges in $\cM_1$ to $A(t)$. Then,
$M_n(t) = \dfrac{1}{n^{1/2}}\sum_{k=1}^{N_{a_n t}}Y_k = \sigma \cW_{n}\circ A_n(t) $ converges to $\sigma \cW\circ A_t$ f.d.d., but the convergence in not necessarily with respect to $\cM_1$. It then follows that
$\disp
\frac{X_{a_nt}-\mu n A_n(t) }{n^{1/2}}  = \sigma \cW_{n}\circ A_n(t) +o_P(1)$.
This can be written as
\begin{equation}\label{eq:ane}
X_{a_n t} =   n^{1/2} \sigma \cW_{n}\circ A_n(t) + n \mu A_n(t) +o_P(n^{1/2}) .
\end{equation}
Processes of the form  $Y = \sigma W\circ A+ \mu A$,  where $A$ is a non-negative stochastic process independent of the Brownian motion $W$ have been proposed by \citet{Ane/Geman:2000} to model the behaviour of stock returns. A particular case is the famous Variance Gamma model \citep{Madan/Carr/Chang:1998}, where $A$ is a Gamma process. It follows that these processes can appear as limiting cases of \eqref{eq:ane}.
If in addition, $A$ is deterministic, continuous, and
$\disp  \dA_n(t) = \sn\left\{A_n(t)-A(t)\right\} = \dfrac{N_{a_n t}- n A(t)}{n^{1/2}}  \stackrel{\cJ_1}{\rightsquigarrow} \dA(t)$,
then
$$
n^{1/2}\left\{ \dfrac{ X_{a_n t} }{n}-\mu A(t)\right\}  = \sigma \cW_n\circ A_n(t)  + \mu \dA_n(t)+ o_P(1).
$$
In most cases, $\dA$ will be a Brownian motion (up to a constant), independent of $\cW$.  In what follows, we assume that $N_t$ be a renewal counting process defined by $\{N_t \ge k\} = \{S_k \le t\}$, where  $S_k = \sum_{j=1}^k \tau_k$, and the $\tau_k$s are iid positive random variables. The asymptotic behaviour of $N$ is determined by the asymptotic behaviour of $S$.

\subsubsection{First case: $S_n$ has finite variance}
Suppose $\tau_k$ has mean $c_1$ and variance $\sigma_1^2$. Then, taking $a_n=n$, one gets that $A(t) = \dfrac{t}{c_1}$, $\dA = -\dfrac{\sigma_1}{c_1^{3/2}}\dW$, where $\dW$ is a Brownian motion independent of $\cW$.
It then follows from \citet[Theorem 2.1]{Remillard/Vaillancourt:2024a} (stated as Theorem \ref{thm:jumps_vanish} here) that
$n^{1/2}\left\{\dfrac{X_{a_n t}}{n}- t\dfrac{\mu}{c_1} \right\}$ converges in $\cC$ to $\tilde\sigma \tilde \cW(t)$, where $\tilde \cW$ is a Brownian motion and $\tilde\sigma = \left(\dfrac{\sigma^2}{c_1}+ \dfrac{\mu^2 \sigma_1^2}{c_1^3}\right)^{1/2}$. This corrects formula (16) in \cite{Chavez-Casillas/Elliott/Remillard/Swishchuk:2019}, where the factor $\sigma_1^2$ is missing. This is also Theorem 7.4.1 in \cite{Whitt:2002}.

\subsubsection{Second case: $S_n$ is in the domain of attraction of a stable process of order $\alpha<1$}
Setting $a_n=n^{1/\alpha}$, assume that  $\cS_n(t) = S_{\lfloor n t\rfloor}/a_n$  converges in $\cJ_1$ to $\cS(t)$.
It follows that $A_n(t) = \dfrac{N_{n^{1/\alpha} t}}{n}$ converges in $\cM_1$ to $A(t) = S^{-1}(t)$ since
$S_n^{-1}(t) = \dfrac{1+N_{n^{1/\alpha} t}}{n}$. Then,
$M_n(t) = \dfrac{1}{\sigma n^{1/2}}\sum_{k=1}^{N_{a_nt}}V(\xi_k) = \cW_{n}\left(N_{a_n t}/n\right)$
converges to $\cW\circ A_t$ f.d.d., but the convergence in not with respect to $\cM_1$.
\BR{Since $A_n$ is not continuous, Proposition \ref{prop:Whitt_composition} is not applicable to prove
$\cW_{n}\circ A_n \stackrel{\cM_1}{\not\rightsquigarrow}  \cW\circ A$.
The linear interpolate $\tilde A_n$ between successive values of $A_n$, which has strictly increasing and continuous trajectories,
satisfies $A_n\le \tilde A_n\le S_n^{-1}$ and $\tilde A_n\stackrel{\cM_1}{\rightsquigarrow} A$.
Proposition \ref{prop:Whitt_composition} implies $\cW_{n}\circ \tilde A_n \stackrel{\cM_1}{\not\rightsquigarrow}  \cW\circ A$.
Since $\tilde A_n - A_n \stackrel{\mathcal{C}}{\rightsquigarrow} 0$,
$\cW_{n}\circ \tilde A_n - \cW_{n}\circ A_n \stackrel{\mathcal{C}}{\rightsquigarrow} 0$ and
$\cW_{n}\circ A_n \stackrel{\cM_1}{\not\rightsquigarrow}  \cW\circ A$ both ensue.}
However,  it follows from Theorem \ref{thm:comp} that $\cW_n\circ A_n \stackrel{\cL^2_w}{\rightsquigarrow}\cW\circ A$.
Note that it also follows that $\dfrac{X_{a_n t}}{n} \stackrel{\cM_1}{\rightsquigarrow} \mu A(t)$.

\begin{rem}\label{rem:error} There have been many incorrect statements about convergence of $A_n$ in the literature. For example, Corollary 3.4 of \cite{Meerschaert/Scheffler:2004} about the $\cJ_1$-convergence of $A_n$  is not true. The convergence is in $\cM_1$, not $\cJ_1$. The authors quoted an erroneous result in Theorem 3 of \cite{Bingham:1971}, where the error in the proof of the latter result comes from the incorrect inequality in Equation (9) in \cite{Bingham:1971}. As stated in that article,
while it is true that
$\frac{1}{2}\omega_\dJ\left(x,\frac{T}{2k},T\right) \le D(x,k,T)$, where $\disp D(x,k,T) = \max_{1\le r\le k}|x(rT/k)-x((r-1)T/k)|$, it is not true that $D(x,k,T) \le \omega_\dJ\left(x,\frac{T}{k},T\right)$ as claimed, even if $x$ is monotone. If $X\in D$, then $\limsup_{k\to\infty}P(D(k,x,T)>\epsilon) >0$ unless $P(X\in C)=1$. In fact, for any $x\in D$,
$\disp \limsup_{k\to\infty}D(k,x,T) =\sup_{0<t\le T}|\Delta x(t)|$, and if $x$ is nondecreasing, then
$$
\omega_\dJ\left(x,\delta,T\right) \le \sup_{\delta\le t \le T-\delta}\min\{x(t)-x(t-\delta),x(t+\delta)-x(t)\}.
$$
Next, Theorem 4.2 in \cite{Meerschaert/Scheffler:2004}, also quoted in  \cite{Scalas/Viles:2012,Scalas/Viles:2014}, is also incorrect. The authors concluded
 that since the jumps of $\cW$ and $\cS$ are not simultaneous by independence, it follows that $A$ is strictly increasing at points on discontinuity of $\cW$, then $\cW_n\circ A_n \stackrel{\cJ_1}{\rightsquigarrow}\cW\circ A$.
 However, this is an incorrect use of  \citep[Theorem 13.2.4]{Whitt:2002}.
 One cannot even conclude that $\cW_n\circ A_n \stackrel{\cM_1}{\rightsquigarrow}\cW\circ A$.
\end{rem}

\subsubsection{Third case: $S_n$ is in the domain of attraction of a stable process of order $\alpha=1$}
Next, if  $\disp \lim_{x\to\infty}xP(\tau_1>x)\to c_1\in (0,\infty)$, then $S_\nt/n\log{n} \to c_1$ and $N_{nt}/\left(n/\log{n}\right)\to t/c_1$, using   $E\left(e^{-s\tau_1}\right) = 1 + c_1 s\log{s} + o(s\log{s})$, as $s\to 0$, from the proof of \cite[Proposition A.1]{Chavez-Casillas/Elliott/Remillard/Swishchuk:2019}.
In particular, if $a_n = n\log{n}$, then $N_{a_n t}/n$ converges in probability to $\frac{t}{c_1}$. If in addition  $E\left(e^{-s\tau_1}\right) = 1 + c_1 s\log{s}+c_2 s + o(s)$, as $s\to 0$, then  $S_n$ is in the domain of attraction of a stable law of order 1, since
$\frac{S_{\lfloor nt \rfloor}}{n}-c_1 t \log{n} \rightsquigarrow c_1 \cS_t - c_2 t $ in $\cJ_1$,
 where $\cS$ is a stable process of index $1$.
 This clarifies conditions under which the last case in  \cite[Section 4.2.2]{Chavez-Casillas/Elliott/Remillard/Swishchuk:2019} holds true.
 As a result, as $n\to\infty$, $E\left\{e^{-s\cS_n(t)}\right\}\to e^{t\left(c_1 s\log{s}+c_2 s\right)}$, so
$\disp
E\left\{e^{is \cS(t)}\right\} = e^{-ic_2 st +tc_1 \psi(s)}$,
where $\psi(s) = -|s|\left\{\dfrac{\pi}{2} +i{\rm sgn}(s)\log{|s|})\right\}$. See Formula (3.19) in \cite{Feller:1971}. Hence,
$\disp
\log{n}\left\{ \dfrac{X_{a_n t}}{n}-\mu t/c_1\right\} \stackrel{\cM_1}{\rightsquigarrow} -\mu\left(\cS_t -  t \dfrac{c_2}{c_1}\right)$,
since $\log{n}\left\{A_n(t) - t/c_1\right\} \stackrel{\cM_1}{\rightsquigarrow} t \dfrac{c_2}{c_1}-\cS_t $.

\subsubsection{Fourth case: $S_n$ is in the domain of attraction of a stable process of order $\alpha\in (1,2)$}
Finally,  if $1<\alpha<2$, then we have the following result \citep[Theorem 7.4.3]{Whitt:2002}: setting $a_n=n$, one gets
$$
n^{1-1/\alpha} \left(\dfrac{X_{nt}}{n} -t   \dfrac{\mu}{\mu_1}\right) = n^{-\left(\frac{1}{\alpha}-\frac{1}{2}\right)}\cW_n\circ A_n(t) + \mu n^{1-1/\alpha}\left(A_n(t)-\frac{t}{\mu_1}\right)\stackrel{\cM_1}{\rightsquigarrow} - \mu \dfrac{\cS_t}{\mu_1^{1+1/\alpha}},
$$
since $\cS_n(t) = n^{-1/\alpha}\left( S_{\lfloor nt\rfloor} - n t \mu_1\right)\stackrel{\cJ_1}{\rightsquigarrow} \cS_t$, where $\mu_1 = E(\tau_1)$. Note that $\cS_t$ has a stable distribution normalized so that
$\disp \lim_{x\to\infty} x^\alpha P(\tau_1> x) = \lim_{x\to\infty} x^\alpha P(\cS_1> x)= c_1$.


\subsection{Occupation times for planar Brownian motion}\label{ssec:occupation_times}

The main results of this section are: a complete proof of Theorem 3.1 in  \citet{Kasahara/Kotani:1979}
and a correction to their Theorem 3.2.\\

Let $B$ be a planar Brownian motion and let $V: \dR^2 \mapsto \dR$ be continuous with compact support
and set $\bar V= \int_{\dR^2}V(x)dx$. Then
$F(x) = \pi^{-1}\int_{\dR^2}V(y) \log|y-x|dy$ is bounded and
$
M_t = F(B_t) - \int_0^t V(B_s)ds$ is a martingale with $[M]_t = \int_0^t |\nabla F(B_s)|^2ds$
by It\^o's formula, since $\dfrac{1}{2}\Delta F = V$ holds \citep[Theorems 6.21 and 10.2]{Lieb/Loss:2001}.
So, under rescaling, the martingale $M$ and the occupation time $\int_0^t V(B_s)ds$ have the same asymptotic behavior.
Hypothesis \ref{hyp:an_unbounded}.a holds, provided $V$ is not identically null, since $\nabla F$ is continuous
and also not everywhere null, entailing $[M]_\infty = \infty$. Since $M$ is continuous,
Hypothesis \ref{hyp:bm_limit} is also satisfied for all rescalings.
Let $m(t)$ denote a numerical function that goes to $\infty$ with $t$, and define the continuous martingale
$\tilde M_n(t) =\{\log{m(n)}\}^{-1/2} \int_0^{m(nt)}\nabla F(B_s)\cdot dB_s$ with quadratic variation
$\tilde A_n(t)=[\tilde M_n]_t =\{\log{m(n)}\}^{-1} \int_0^{m(nt)} |\nabla F(B_s)|^2ds$.
The common formulations involve rescalings $m(t)=t$ and $m(t)=te^{2t}$ (ensuring that $\tilde A_n(0)=0$).

\subsubsection{First case: $m(t)=t$}
Here, removing the superscript on $\tilde A_n$ in order to avoid confusion for the reader
between the treatment of the two rescaled sequences,
$A_n(t) = \disp \{\log{n}\}^{-1} \int_0^{nt} |\nabla F(B_s)|^2ds$ converges f.d.d. to $A$, where $A(0)=0$ and
$A(t) =A(1) \stackrel{Law}{=} \disp \dfrac{\cE}{2\pi}\int_{\dR^2} |\nabla F(x)|^2 dx  = \dfrac{\cE}{2\pi}c_V$ for every $t>0$,
where ${\cE}$ is an Exponential random variable with mean $1$, and where
 \begin{equation}\label{cVisnablaF}
c_V := -\frac{2}{\pi}\int\int_{\dR^4} V(x)V(y)\log|x-y| dx dy = \int_{\dR^2} |\nabla F(x)|^2 dx.
\end{equation}
The proof of \eqref{cVisnablaF} is given in \ref{pf:cVisnablaF}.
The proof of $A_n \stackrel{f.d.d.}{\rightsquigarrow} A$ follows from the fact that for any $0< s \le t$,
$A_n(t)-A_n(s) \stackrel{Law}{\rightsquigarrow} 0$, since $0\le EA_n(t)-EA_n(s)\to 0$, using \cite{Darling/Kac:1957}.
As a result,  $A$ is clearly not right-continuous at $0$ and therefore does not belong to $D$.
It also fails to verify another requirement of Theorem \ref{thm:clt_mart}, namely $A(t)\to\infty$ as $t\to\infty$
in Hypothesis \ref{hyp:an_unbounded}.

\begin{rem}
Proving the convergence in distribution of $A_n(t)$ was achieved in stages. \citet{Kallianpur/Robbins:1953} claimed that $[M]_t/\log{t}$ converges in law to an Exponential distribution. In fact they never proved this result. They refer to \citet{Robbins:1953},
where a result on sums of independent random variables on the plane is claimed but not proved.
\citet{Kallianpur/Robbins:1953} actually appeared before \citet{Robbins:1953}.
It seems that the result claimed by \citet{Robbins:1953} was in fact proven in \citet{Darling/Kac:1957},
while \citet{Erdos/Taylor:1960} proved a special case involving the number of visits at $0$ for symmetric random walks.
In fact, \cite{Darling/Kac:1957},  when $V$ is a non-negative function, showed that  for any given $t>0$,
not only does $A_n(t)$ converge in law, but all moments of $A_n(t)$ also converge to those of the limiting Exponential distribution.
The general case appeared in \citet{Yamazaki:1992} and \citet[Lemma II.6 and ensuing remarks]{LeGall:1992},
after some preliminary work by others, notably \citet{Erdos/Taylor:1960} and \citet{Kasahara/Kotani:1979}.
\end{rem}

\subsubsection{Second case: $m(t)=te^{2t}$}
This scaling was proposed by \citet{Kasahara/Kotani:1979}. In fact,
\citet{Kasahara/Kotani:1979} showed that, if $G$ is a continuous positive function with $\bar G = \int_{\dR^2}|x|^\alpha G(x)dx < \infty $, for some $\alpha>0$, then
$\dfrac{1}{n}\int_0^{m(nt)}G(B_s)ds$ is equivalent in law to
$\disp
H_n(t) = \frac{1}{n}\int_0^{S^{-1}\{m(nt)\}} f(\beta_s,\theta_s) ds$,
where $\beta$ and $\theta$ are two independent Brownian motions, $S_t = \int_0^t e^{2\beta_u}
du$, and $f$ is chosen so that $\bar f = \int_{-\infty}^\infty \int_0^{2\pi}f(x,u)dx du = \bar G$. Then, setting
$T_n(t) = \dfrac{1}{n}m^{-1}\left\{S\left(n^2 t\right)\right\}$,
one gets $H_n\circ T_n(t) = \dfrac{1}{n}\int_0^{n^2 t}f(\beta_s,\theta_s) ds$.
As a result, they proved that
\begin{equation}\label{eq:K&K}
\left(H_n\circ T_n(t), \beta(n^2t)/n,T_n(t)\right)\stackrel{\mathcal{C}}{\rightsquigarrow} \left(2\bar G \ell, b,\dM\right),
\end{equation}
where $b$ is a Brownian motion, $\dM_t = \disp\sup_{s\le t}b(s)$, and $\ell_t $ is the local time at $0$ of $b$.
It then follows that $T_n^{-1}\stackrel{\cM_1}{\rightsquigarrow}\dM^{-1}$, which is a L\'evy process. In their Theorem 3.1,
 \citet{Kasahara/Kotani:1979} stated that $H_n\stackrel{\cM_1}{\rightsquigarrow}H = 2 \bar G \ell \circ M^{-1}$ without proving it,
 stopping at \eqref{eq:K&K}.  To complete their proof, one may apply Proposition \ref{prop:Whitt_composition}
 with $x_n = H_n\circ T_n$, $y_n = T_n^{-1}$, $x = 2\bar G \ell$, and $y= \dM^{-1}$ to conclude that
 $H_n\stackrel{\cM_1}{\rightsquigarrow}H = x\circ y = 2\bar G \ell\circ \dM^{-1}$, since $\ell$ is monotone.
\begin{rem}
It does not seem that such a  result on the convergence of composition was available to \citet{Kasahara/Kotani:1979},
since its first version appeared in \cite[Proposition 3.1]{Yamazaki:1992}.
\end{rem}

Next, in their Theorem 3.2,  \citet{Kasahara/Kotani:1979} stated, without proof, that if $\bar V=0$,
then $\tilde M_n(t) = \frac{1}{\sn}\int_0^{m(nt)}V(B_s)ds$ converges in  $\cM_1$ to $c_V \beta_2\circ H$, where $c_V=\dfrac{\bar f}{2\pi}$,
$H = 2 \bar G \ell \circ M^{-1}$ and $\beta_2$ is a Brownian motion independent of $b$ and $H$. To try to prove this result, note that
$\frac{1}{\sn}\int_0^{m(nt)}V(B_s)ds$ is equivalent in law to
$\disp M_n(t) = \frac{1}{\sn}\int_0^{S^{-1}\{m(nt)\}} f(\beta_s,\theta_s) ds$,
where $\bar f = \bar V=0$. The latter is almost a martingale due to the decomposition described previously.
Applying Theorem \ref{thm:jumps_vanish}, one gets that, for some $\gamma>0$,
$$
\left(M_n\circ T_n(t), H_n\circ T_n(t), \beta(n^2t)/n,T_n(t)\right)
\stackrel{\mathcal{C}}{\rightsquigarrow}\left(\gamma\beta_2\circ \ell, \bar G \ell, b,\dM\right),
$$
where $\beta_2$ is a Brownian motion independent of $\ell$. Because of Proposition \ref{prop:Whitt_composition},
$M_n = M_n\circ T_n\circ T_n^{-1}$ does not $\cM_1$-converge to $\gamma\beta_2\circ \ell\circ \dM^{-1}$ in $D$,
because $\beta_2$ is not monotone with probability one on the sets $[ \ell\circ \dM^{-1}(t-), \ell\circ \dM^{-1}(t)]$. Therefore, Theorem 3.2 in
\citet{Kasahara/Kotani:1979} is incorrect.
However, since the points of discontinuity of $T^{-1}$ are at most countable and not fixed, it follows from
\citet[Theorem 12.4.1]{Whitt:2002} that $M_n \stackrel{f.d.d.}{\rightsquigarrow} M_\infty = \gamma \beta_2\circ \ell\circ \dM^{-1}$.
This result was also proven in \citet{Csaki/Foldes/Hu:2004}, using strong approximations.
Next, the f.d.d. convergence implies that  $\int_0^t M_n(s)ds \stackrel{f.d.d.}{\rightsquigarrow} \int_0^t M_\infty(s)ds$,
and hence this completes the proof of $M_n \stackrel{\cL^2_w}{\rightsquigarrow} M_\infty$,
in this second case where $m(t)=te^{2t}$,
by an argument similar to the one used to prove Theorem \ref{thm:comp}.
In fact, for any choice of finite partition $0=t_0< t_1< t_2 \cdots <t_m=t$, $m\ge1$ and $n=\infty,1,2,\dots$, put
$$
Z_{n,m}(t):= \sum_{k=1}^m M_n(t_k)\left\{(t_k\wedge t)-(t_{k-1}\wedge t)\right\}.
$$
As $n\to\infty$, $Z_{n,m}(t)\stackrel{Law}{\rightsquigarrow}Z_{\infty,m}(t)$, while, as $m\to\infty$,
\begin{multline*}
E\left\{(Z_{n,m}(t) - \hbox{$\int_0^t$} M_n(s)ds)^2\right\} \\
\le \sum_{k=1}^m\left\{(t_k\wedge t)-(t_{k-1}\wedge t)\right\}^2\left\{E[M_n]_{t_k\wedge t}-E[M_n]_{t_{k-1}\wedge t}\right\}\\
\le \sup_k(t_k-t_{k-1})^2\sup_nE[M_n]_t\to0,
\end{multline*}
through the selection of any sequence of partitions $\{t_{k,m};k\le m\}$ with vanishing mesh, i.e.,
such that $\sup_k(t_{k,m}-t_{k-1,m})\to0$ as $m\to\infty$.


\appendix

\section{Miscellanea}\label{app:remvail}

Let $D= D[0,\infty)$ be the space of $\dB$-valued c\`adl\`ag trajectories (right continuous with left limits everywhere), for some separable Banach space $\dB$, i.e., $\dB$ is a complete separable normed linear space with norm $\|\cdot\|$.
$\dB$ will always be specified implicitly by the context at hand (usually $d$-dimensional Euclidean space $\dR^d$)
so the subscript $\dB$ is omitted. All processes considered in this section have their trajectories in $D$
and are adapted to a filtration $\dF = (\cF_t)_{t\ge 0}$ on a probability space
$(\Omega,\cF,P)$ satisfying the usual conditions (notably, right continuity of the filtration and completeness).
Trajectories in $D$ are usually noted $x(t)$ but occasionally $x_t$.

On $\dB^3$, set $\dC(x_1,x_2,x_3)=\|x_3-x_1\|$, $\dJ(x_1,x_2,x_3) = \|x_2-x_1\|\wedge \|x_2-x_3\|$, where $k\wedge\ell = \min(k,\ell)$
and let $\dM(x_1,x_2,x_3)$ be the minimum distance between
$x_2$ and the Banach space segment $[x_1,x_3]:=\{\lambda x_1+(1-\lambda)x_3\in \dB; \lambda\in[0,1]\}$. Then
$\dM(x_1,x_2,x_3) = 0 $ if $x_2 \in [x_1,x_3]$ and otherwise $\dM(x_1,x_2,x_3) \le \dJ(x_1,x_2,x_3)$.

For $\dH=\dM$, $\dH=\dJ$ or $\dH=\dC$, $T>0$, and $x \in D$, set
$$
\omega_\dH(x,\delta,T) = \sup_{0 \le t_1 < t_2 < t_3\le T, \; t_3-t_1<\delta}
\dH\{x(t_1),x(t_2),x(t_3)\}.
$$

Using the terminology in \citet[Theorem 3.2.2]{Skorohod:1956}, for each $\dH=\dM$, $\dH=\dJ$ or $\dH=\dC$,
a sequence of $D$-valued processes $X_n$ is $\mathcal{H}_1$-tight (respectively $\mathcal{M}_1$-tight,
$\mathcal{J}_1$-tight or simply $\mathcal{C}$-tight) if and only if
\begin{enumerate}
\item[(i)]
for every $t$ in an everywhere dense subset of $[0,\infty)$ that includes the origin, the marginal distributions of $X_n(t)$ are tight,

\item[(ii)]
 for any $\epsilon>0$, and $T>0$,
\begin{equation}\label{eq:modulus2}
\lim_{\delta\to 0}\limsup_{n\to\infty}P\{\omega_\dH(X_n,\delta,T) > \epsilon\}=0.
\end{equation}
\end{enumerate}

A sequence of processes $X_n$ converges weakly under the $\mathcal{H}_1$-topology on $D$ to $X$,
denoted $X_n\stackrel{\mathcal{H}_1}{\rightsquigarrow} X$ when $\dH=\dJ$ or $\dH=\dM$, if and only if
it is $\mathcal{H}_1$-tight and the finite dimensional distributions of $X_n$ converge to those of
$X$, denoted $X_n \stackrel{f.d.d.}{\rightsquigarrow} X$, over a dense set of times in $[0,\infty)$ containing $0$.
Similarly write $X_n\stackrel{\mathcal{C}}{\rightsquigarrow} X$ for weak convergence under the
$\mathcal{C}$-topology on $D$, that of uniform convergence over compact time sets, defined by taking $\dH=\dC$.
While we are on the subject of notation, equality in law or equidistribution is denoted by $\stackrel{Law}{=}$,
convergence in the $p^{th}$ mean by $\stackrel{L^p}{\rightsquigarrow}$,
in probability by $\stackrel{Pr}{\rightsquigarrow}$, in law by $\stackrel{Law}{\rightsquigarrow}$ and
almost sure convergence by $\stackrel{a.s.}{\rightsquigarrow}$.

Complete coverage of the $\cJ_1$-topology can be found in \citet{Ethier/Kurtz:1986} and \citet{Jacod/Shiryaev:2003},
where $\dB$ is actually allowed be a Polish space; while \citet{Whitt:2002} is our main reference on the $\mathcal{M}_1$-topology.
Neither topology turns $D$ into a topological vector space, that is to say, neither is compatible with the linear structure inherited
on $D$ from the Banach space $\dB$, which would have made the sum a continuous operator. Consequently, when considering vectors of processes valued in a product $\prod_{i=1}^d \dB_i$ of Banach spaces, neither topology offers the equivalence
between $\mathcal{H}_1$-convergence of sequences of probability measures on $D$ and coordinatewise $\mathcal{H}_1$-convergence,
an equivalence which holds for $\mathcal{C}$-convergence.
Similarly $\mathcal{H}_1$-tightness cannot be treated coordinatewise; it does however ensue from the
$\mathcal{H}_1$-tightness of each coordinate plus that of each pairwise sum of coordinates --- $d^2$ verifications are thus required.
This fact precludes immediate extensions of some results from the real line to higher dimensional spaces, but there are exceptions,
as will be seen next.

\begin{rem}\label{rem:producttopology}
All topological statements in this paper involving multiple coordinate functions or processes with $\dB$-valued c\`adl\`ag trajectories,
are stated and proved for the space $D=D([0,\infty):\dB)$, with Banach space product $\dB:=\prod_{i=1}^d\dB_i$ equipped with the norm
$| \cdot |_{\dB}:=\vee_{i=1}^d | \cdot |_{\dB_i}$. Any conclusion thus stated remains valid for the weaker
product topology for the space $\prod_{i=1}^d D([0,\infty):\dB_i)$ without further ado. This is the case whether the topology involved is
$\mathcal{C}$, $\cL^2_w$, $\mathcal{M}_1$ or $\mathcal{J}_1$. On those rare occasions when the latter product topology
is used, it is done explicitly by displaying the product space, such as in Proposition \ref{prop:Whitt_composition}.
\end{rem}

The following inequalities link the moduli associated with $\mathcal{H}_1$-topologies,
for both $\dH=\dJ$ and $\dH=\dM$, with that of the $\mathcal{C}$-topology,
thus providing a first situation when coordinatewise arguments do work.

\begin{lem}\label{lem:BasicInequality}
For any pair of functions $X, Y\in D$ there holds, for every choice of $\delta>0$ and $T>0$,
$\omega_\dH(X+Y,\delta,T)\le \omega_\dH(X,\delta,T)+\omega_\dC(Y,\delta,T)$.
For sequences of $D$-valued processes, if $X_n$ is $\mathcal{H}_1$-tight and
$Y_n$ is $\mathcal{C}$-tight, then $X_n+Y_n$ and $(X_n,Y_n)$ are also $\mathcal{H}_1$-tight.
Similarly if $X_n\stackrel{\mathcal{H}_1}{\rightsquigarrow} X$, $Y_n\stackrel{\mathcal{C}}{\rightsquigarrow} Y$, and
$(X_n,Y_n)\stackrel{f.d.d.}{\rightsquigarrow} (X,Y)$ then $X_n+Y_n\stackrel{\mathcal{H}_1}{\rightsquigarrow} X+Y$
and $(X_n,Y_n)\stackrel{\mathcal{H}_1}{\rightsquigarrow}(X,Y)$.
\end{lem}

\begin{proof} Given two triplets $(x_1,x_2,x_3), (y_1,y_2,y_3)\in \dB^3$, there is a $\lambda\in[0,1]$ achieving the minimum for
$\dM(x_1,x_2,x_3)=\|x_2-\{\lambda x_1+(1-\lambda)x_3\}\|$ and therefore
\begin{eqnarray*}
& & \dM(x_1+y_1,x_2+y_2,x_3+y_3) \\
& \le & \|x_2+y_2-\{\lambda (x_1+y_1)+(1-\lambda)(x_3+y_3)\}\| \\
& \le & \|x_2-\{\lambda x_1+(1-\lambda)x_3\}| + |y_2-\{\lambda y_1+(1-\lambda)y_3\}\| \\
& = & \dM(x_1,x_2,x_3) +  \|y_2-\{\lambda y_1+(1-\lambda)y_3\}\| \\
& \le& \dM(x_1,x_2,x_3) +  \|y_2-y_1\| \vee \|y_2-y_3\|,
\end{eqnarray*}
this last inequality using convexity of the norm and $k\vee\ell = \max(k,\ell)$.
Restricting to $\lambda\in\{0,1\}$ instead of $[0,1]$ yields the same inequalities with $\dJ$ in place of $\dM$.
\end{proof}

The following result provides necessary and sufficient conditions under which the
composition mapping $\circ$ is $\mathcal{M}_1$-continuous, when $B=\dR^k$.
Let $D_{\uparrow}\subset D$ be the subspace of those non-decreasing non-negative $y$ such that $y(t) \to\infty$ as $t\to\infty$.

\begin{prop}\label{prop:Whitt_composition}
Suppose that $(x_n, y_n) \stackrel{}{\to} (x, y) $ in $D \times  D_{\uparrow}$, where $x$ is continuous and $D$ is equipped with the
$\cM_1$-topology. Assume also that every $y_n$ is continuous and strictly increasing.
Then $\disp x_n \circ y_n \stackrel{}{\to}  x\circ  y$ in $D$ if and only if $x$ is monotone on $[y(t-), y(t)]$,
for any $t\in {\rm Disc}(y)$.
\end{prop}

\begin{proof} Sufficiency follows from \citet[Theorem 13.2.4]{Whitt:2002} so we only prove necessity.
Suppose $y_n \to y$ in $\cM_1$, $y_n$ continuous and strictly increasing.  Assume that $t\in {\rm Disc}(y)$, and let $t_{1,m}$ be sequence increasing to $t$ with $t_{1,m} \not \in {\rm Disc}(y)$.
Further let  $t_{m,3}$ be sequence decreasing to $t$ with $t_{m,3} \not \in {\rm Disc}(y)$. Finally, for a given $\lambda \in (0,1)$, let $t_{n,2,m}$ be the unique point such that
$y_n(t_{n,2,m}) = \lambda y_n(t_{1,m})+(1-\lambda)y_n(t_{3,m})$. For a given $\delta>0$, $t-\delta < t_{1,m} < t_{m,3} < t+\delta$ if $m$ is large enough. By \cite[Theorem 12.5.1 (v)]{Whitt:2002}, since $t_{1,m},t_{3,m}$ are continuity point of $y$, it follows that
$\disp \lim_{\delta_1\to 0}\limsup_{n\to\infty}\omega_\dC(y_n-y,t_{j,m},\delta_1) = 0$, $j\in\{1,3\}$;
in particular, $y_n(t_{j,m})\to s_{j,m} = y(t_{j,m})$ as $n\to\infty$.
Thus, $y_n(t_{n,2,m}) \to s_{2,m} = \lambda y(t_{1,m})+(1-\lambda)y(t_{3,m})$. If $x_n\to x$ in $\cM_1$ with $x$ continuous,
it follows from \cite[Theorem 12.5.1 (v)]{Whitt:2002} that $x_n\circ y_n(t_{n,j,m})\to x(s_{j,m})$, $j\in\{1,2,3\}$. Therefore,
$$
\limsup_{n\to\infty}\omega_\dM(x_n\circ y_n,\delta,T)\ge
\dM\left\{x(s_{1,m}),x(s_{2,m}),x(s_{3,m})\right\}
$$
if $m$ is large enough. By letting $m\to\infty$, one gets $s_{1,m}\to s_1 = y(t-)$, $s_{3,m}\to s_3 = y(t)$ and $s_{2,m}\to \lambda y(t-)+ (1-\lambda)y(t)$.  As  a result,
\begin{multline*}
\lim_{\delta \to 0} \limsup_{n\to\infty}\omega_\dM(x_n\circ y_n,\delta,T) \\
 \ge  \sup_{t\in {\rm Disc}(y), t <T}\sup_{y(t-)  \le s \le y(t)}\dM\left[x\{y(t-)\},x(s),x\{y(t)\}\right]>0,
\end{multline*}
unless $x$ is monotone on each interval $[y(t-),y(t)]$.
\end{proof}

From hereon we focus exclusively on the real-valued functions and processes so $\dB=\dR$.
Some observations for sequences of $D$-valued non-decreasing processes come next.
For any non-decreasing non-negative function $A\in D$, denote its inverse by $\tau(s) = \inf\{t\ge0; A(t)>s\}$.
Let $D_{\uparrow}^0$ be the subspace of those trajectories in $D_{\uparrow}\subset D$ where $\tau(0)=0$ as well.

\begin{prop}\label{prop:inversecontinuity}
The inverse map $A\mapsto\tau$ is a well defined bijective mapping from $D_{\uparrow}$ into itself, such that there holds
$A\circ\tau\circ A=A$ provided $A\in D_{\uparrow}$ is either continuous everywhere or strictly increasing everywhere.
Both $A\mapsto\tau$ and the reverse map $\tau\mapsto A$ are $\mathcal{M}_1$-continuous
when restricted to $D_{\uparrow}^0$; this is not necessarily the case without this additional restriction, not even on $D_{\uparrow}$.
\end{prop}

\begin{proof}
The first statement combines \citet[Corollary 13.6.1]{Whitt:2002} together with \citet[Lemma 13.6.5]{Whitt:2002};
the second one is \citet[Theorem 13.6.3]{Whitt:2002} essentially.
\end{proof}

\begin{rem}\label{rem:M1increasing} A sequence of $D$-valued non-decreasing processes verify
$A_n \stackrel{\mathcal{M}_1}{\rightsquigarrow} A$
if and only if $A_n \stackrel{f.d.d.}{\rightsquigarrow} A$, since they all satisfy $\omega_\dM(A_n,\delta)=0$ for any $\delta>0$.
Similarly, the $A_n$ are $\mathcal{M}_1$-tight if and only if their finite dimensional distributions are tight; since each $A_n$ is
non-decreasing, this last condition is equivalent to the tightness of $(A_n(T))_{n\ge 1}$ for every $T>0$, provided the $A_n$ are
non-negative (or at least uniformly bounded below). Furthermore, the
sum $A_n+A'_n$ of two $\mathcal{M}_1$-tight sequences of non-negative, non-decreasing processes is also $\mathcal{M}_1$-tight,
hence so are $(A_n,0)$, $(0,A'_n)$ and finally $(A_n,A'_n)$ as well. If in addition there holds
$A_n \stackrel{\mathcal{M}_1}{\rightsquigarrow} A$ and $A'_n \stackrel{\mathcal{M}_1}{\rightsquigarrow} A'$, then
there ensues $A_n+A'_n \stackrel{\mathcal{M}_1}{\rightsquigarrow} A+A'$ and
$(A_n,A'_n) \stackrel{\mathcal{M}_1}{\rightsquigarrow} (A,A')$.
\end{rem}

More can be said when the prospective limit is continuous. The following result is taken from \citet[Proposition A.4]{Remillard/Vaillancourt:2024a}.

\begin{prop}\label{prop:Cincreasing}
Let $D$-valued non-decreasing processes $A_n$ and some continuous process $A$
be such that $A_n \stackrel{f.d.d.}{\rightsquigarrow} A$. Then $A_n$ is $\mathcal{C}$-tight
and $A_n \stackrel{\mathcal{C}}{\rightsquigarrow} A$.
\end{prop}

The following technical result ensues. Once again $M_n$ is a sequence of $D$-valued square integrable
$\mathbb{F}$-martingales started at $M_n(0)=0$
with quadratic variation $[M_n]$ and predictable compensator $\crochet{ M_n }$.

\begin{prop}\label{prop:Cto0}
$M_n \stackrel{\mathcal{C}}{\rightsquigarrow} 0$,
$[M_n] \stackrel{\mathcal{C}}{\rightsquigarrow} 0$ and
$\crochet{ M_n }  \stackrel{\mathcal{C}}{\rightsquigarrow} 0$ are equivalent.
\end{prop}

\begin{proof} The operator on $D$ defined by mapping $x\in D$ to $t\mapsto\sup_{0\le s\le t}x(s)$,
is a continuous operator in the $\mathcal{C}$-topology. The three statements
$M_n \stackrel{\mathcal{C}}{\rightsquigarrow} 0$, $M_n^2 \stackrel{\mathcal{C}}{\rightsquigarrow} 0$ and
$\sup_{0\le s\le\{\cdot\}}M_n^2(s) \stackrel{\mathcal{C}}{\rightsquigarrow} 0$ are therefore equivalent.
By Proposition \ref{prop:Cincreasing}, it suffices to prove the proposition with each of
$M_n \stackrel{\mathcal{C}}{\rightsquigarrow} 0$,
 $[M_n] \stackrel{\mathcal{C}}{\rightsquigarrow} 0$ and
$\crochet{ M_n } \stackrel{\mathcal{C}}{\rightsquigarrow} 0$,
respectively replaced by $\sup_{0\le s\le t}M_n^2(s) \stackrel{Law}{\rightsquigarrow}0$,
$[M_n]_t\stackrel{Law}{\rightsquigarrow}0$ and
$\crochet{ M_n }_t\stackrel{Law}{\rightsquigarrow}0$, for all $t\ge0$ in all three cases.
Doob's inequality yields
$$
E\{[M_n]_t\} \le E\left\{\sup_{0\le s\le t}M_n^2(s)\right\}\le 4E\{[M_n]_t\}.
$$
Applying Lemma \ref{lem:lenglart} twice, interchanging the roles of $X$ and $Y$, yields the first equivalence for $M_n$ and $[M_n]$.
Equality $E\{\crochet{ M_n }_t\} = E\{[M_n]_t\}$ and another double application of Lemma \ref{lem:lenglart} yields the last one,
this time with $\crochet{ M_n }$ and $[M_n]$.
\end{proof}
Next is a miscellanea of useful results. First comes a CLT in the $\mathcal{C}$-topology,
taken from \citet[Theorem 2.1]{Remillard/Vaillancourt:2024a}. We only state it for $d=1$ here.

\begin{thm}\label{thm:jumps_vanish}
Assume that Hypothesis \ref{hyp:an_unbounded} holds with $A$ continuous everywhere; that
$J_t(M_n) \stackrel{Law}{\rightsquigarrow} 0$ for any $t>0$; that there exists an
$\dF$-adapted sequence of $D$-valued square integrable martingales $B_n$ started at $B_n(0)=0$ so that
\begin{enumerate}
\item $(B_n, A_n) \stackrel{\mathcal{C}}{\rightsquigarrow} (B,A)$ holds, where $B$ is a Brownian motion
with respect to its natural filtration $\dF_B = \{\cF_{B,t}: \; t\ge 0\}$ and $A$ is $\dF_B$-measurable;
\item  $\langle M_n, B_n\rangle_t \stackrel{Law}{\rightsquigarrow} 0$, for any $t\ge 0$.
\end{enumerate}
Then $(M_n, A_n,B_n) \stackrel{\mathcal{C}}{\rightsquigarrow} (M,A,B)$ holds,
where $M$ is a continuous square integrable ${\cF}_{t}$-martingale with respect to (enlarged) filtration
with predictable quadratic variation process $A$.
Moreover, $M = W\circ A$ holds, with $W$ a standard Brownian motion which is independent of $B$ and $A$.

\end{thm}

Next, we have the following observations about L\'evy processes.

The characteristic function of an inhomogeneous L\'evy process $A$ is given by
$$
E\left[\left. e^{i\theta (A_t-A_s)}\right| \cF_s \right] = e^{\Psi_t(\theta)-\Psi_s(\theta)}, \quad 0\le s < t,
$$
where $A$ is $\mathbb{F}$-adapted with $\dF = (\cF_t)_{t\ge 0}$,
$$
\Psi_t(\theta)  = i \theta B_t  -\frac{\theta^2 C_t}{2}
+ \left\{\prod_{0<r\le t} e^{-ir\Delta B_r}\right\}\int_0^t \int \left(e^{i\theta x}-1-i\theta x\I_{\{|x|\le 1\}}\right)\nu(du,dx),
$$
$\Delta B_r = B_r-B_{r-}$, and $\nu$ is a L\'evy measure.
For details see \citet[Theorem II.4.15, Theorem II.5.2 and Corollaries II.5.11,12,13]{Jacod/Shiryaev:2003}.
The characteristics of $A$ are defined to be the deterministic functions $(B,C,\nu)$.

\begin{lem}\label{lem:main}
Suppose $A$ is an $\mathbb{F}$-adapted inhomogeneous L\'evy process with finite variation on any finite interval $[0,T]$ and no Brownian component, i.e., $C_t = 0$ for all $t\ge 0$. Then $A$ is independent of any $\mathbb{F}$-Brownian motion $W$.
\end{lem}

\begin{proof}
If $W$ is an $\mathbb{F}$-Brownian motion, then the martingale $N_1(t) = e^{i \theta_1 W_t+\frac{\theta^2 t}{2}}$ is continuous and square integrable on any finite interval $[0,T]$ and the martingale $N_2(t) = e^{i \theta_2 A_t-\Psi_t(\theta_2)}$ has finite variation $V_T$ on any finite interval $[0,T]$, and $V_T$ is square integrable. Then, from \citet{Applebaum:2004}, one obtains that for any $t\ge 0$, with
$\Delta N(u) = N(u)-N(u-)$,
$$
E[N_1(t) N_2(t)] = 1 + \sum_{ 0<u \le t}E[\Delta N_1(u) \Delta N_2(u) ] = 1.
$$
Hence $E\left[ e^{i \theta_1 W(t)+ i \theta_2 A(t)}\right]  = e^{-\frac{\theta_1^2 t}{2} }e^{\Psi_t(\theta_2)}$, proving that $A(t)$ and $W(t)$ are independent. Using this and the independence of the increments of $A$ and $W$, one may conclude that $A$ and $W$ are independent.
\end{proof}

The proof of the following is in \citet[Lemma I.3.30]{Jacod/Shiryaev:2003}.

\begin{lem}[Lenglart's inequality]\label{lem:lenglart}
Let $X$ be an $\dF$-adapted $D$-valued process. Suppose that $Y$ is optional, non-decreasing,
and that, for any bounded stopping time $\tau$, $E|X(\tau)|\leq E\{Y(\tau)\}$.
Then for any stopping time $\tau$ and all $\varepsilon,\eta >0$,
\begin{itemize}
\item[{a)}]
if $Y$ is predictable,
\begin{equation}\label{eng1}
P(\sup_{s\leq \tau}|X(s)|\geq\varepsilon) \leq \frac{\eta}{\varepsilon} +
P(Y(\tau)\geq\eta).
\end{equation}

\item[b)]
if $Y$ is adapted,
\begin{equation}\label{eng2}
P(\sup_{s\leq \tau}|X(s)|\geq\varepsilon) \leq \frac{1}{\varepsilon}
\left[\eta+ E\left\{J_\tau(Y) \right\}\right] +
P(Y(\tau)\geq\eta).
\end{equation}
\end{itemize}
\end{lem}


We end with a few comments about $\cL^2_{loc}$ which are used in the proofs herein.

A (deterministic) sequence $x_n\in\cL^2_{loc}$ is said to $\cL^2_w$-converge to $x\in\cL^2_{loc}$ if and only if the sequence of scalar products $\int_0^Tx_n(t)f(t)dt \to \int_0^Tx(t)f(t)dt$ converges as $n\to\infty$, for each $f\in\cL^2_{loc}$ and each $T>0$.
This holds if and only if both the following statements hold  for every $T>0$: the sequence $\bigl\{|x_n|_T; n\ge1\bigr\}$ is bounded
and the requirement for convergence of the above scalar products is achieved for a norm dense subset of $\cL^2[0,T]$.
In fact, given any complete orthonormal set $\{f_k\}$ for $\cL^2[0,1]$, the countable family
$\{f_k(\ell^{-1}\cdot); k,\ell=1,2,\ldots\}$, with $f_k(\ell^{-1}t):=0$ when $t>\ell$,
will suffice, since it has a norm dense linear span in each vector space $\cL^2[0,T]$.
Alternatively, given any complete orthonormal set $\{g_k\}$ for Hilbert space $\cL^2[0,\infty)$ or $\cL^2(-\infty,\infty)$,
the same holds with the countable family $\{g_k\}$ --- for instance, in $\cL^2(-\infty,\infty)$,
take the Hermite functions $\{h_k;k\ge0\}$, defined by
$h_k(t):=(\pi^{1/2}2^kk!)^{-1/2}(-1)^ke^{t^2/2}D^k_t(e^{-t^2})$, where $D^k_t$ is the $k^{\hbox{th}}$ derivative with
respect to $t$, which also satisfy the uniform bound $\sup_{k,t}|h_k(t)|=h_0(0)=\pi^{-1/4}$, due to \citet{Szasz:1951}.
The metric $d$ on $\cL^2_{loc}$, defined by
$$
d(x,y):=\sum_{k=0}^\infty \sum_{\ell=1}^\infty 2^{-(k+\ell+1)}\left(1\wedge\left|\int_0^\ell\left\{x(t)-y(t)\right\}h_k(t)dt\right|\right),
$$
is compatible with the $\cL^2_w$-topology on all norm bounded subsets $B$ of either $\cL^2[0,\infty)$ ($\sup_{x\in B}|x|_\infty<\infty$)
or $\cL^2[0,T]$ ($\sup_{x\in B}|x|_T<\infty$), turning all $\cL^2_w$-closed and norm bounded subsets of
either $\cL^2[0,\infty)$ or $\cL^2[0,T]$ into complete separable metric (hereafter Polish) spaces, when equipped with this metric.
In fact, they are compact Polish spaces since all balls in either $\cL^2[0,\infty)$ or $\cL^2[0,T]$
are relatively compact in the $\cL^2_w$-topology --- for a proof see \citet{Lieb/Loss:2001}.


\section{Proofs}\label{app:remvail_prfs}

\subsection{Proof of Theorem \ref{thm:L2characterization}}\label{pf:L2characterization}

The projections $(\pi_{\ell_1,k_1},\pi_{\ell_2,k_2},\ldots,\pi_{\ell_m,k_m}):\cL^2_{loc}\to \dR^m$,
given by setting $\pi_{\ell,k}(f):=\int_0^\ell f(t)h_k(t)dt $ for all $k\ge0$ and $\ell,m\ge1$,
generate the Borel $\sigma$-algebra for metric $d$, and they characterize probability measures on $\cL^2_{loc}$,
just as in the proof of Kolmogorov's existence theorem for probability measures on $\dR^\infty$.
In the light of Lemma \ref{lem:L2tightness}, it suffices to prove that $(i)$ and $(ii)$ together imply,
jointly for all $(k_1,k_2,\ldots,k_m)\in\{0,1,2,\ldots\}^m$ and all $m\ge1$,
$$
\int_0^\cdot X_n(s)\{h_{k_1},h_{k_2},\ldots,h_{k_m}\}(s)ds  \stackrel{f.d.d.}{\rightsquigarrow}
\int_0^\cdot X(s)\{h_{k_1},h_{k_2},\ldots,h_{k_m}\}(s)ds
$$
as $n\to\infty$. The Cram{\'e}r-Wold device, applied to linear combinations of finitely many indicators of finite time intervals,
plus a density argument gives sufficiency. Necessity follows the same argument applied to $\{h_k\}$ to approximate $1\in\cL^2_{loc}$.
\qed

\subsection{Proof of Theorem \ref{thm:comp}}\label{pf:comp}

Setting $|X_n|_{\infty,t} = \sup_{0\le s\le t} |X_n(s)|$, $|X_n\circ Y_n|_T \le \sqrt{T} |X_n|_{\infty,Y_n(T)}$ follows.
Since $Y_n(T)$ is tight for every $T>0$, so is $|X_n\circ Y_n|_T$. Next, under alternatives a) or b),
$V_n = Y_n^{-1} \stackrel{\mathcal{M}_1}{\rightsquigarrow} V = Y^{-1}$, by Proposition \ref{prop:inversecontinuity}.
According to Theorem \ref{thm:L2characterization}, to complete the proof it suffices to prove the convergence of the finite dimensional distributions of the sequence of processes $\int_0^t X_n\circ Y_n(s)ds = Z_n\circ Y_n(t)$ to the process $\int_0^t X\circ Y(s)ds = Z\circ Y(t)$ , where $Z_n(t) = \int_0^t X_n(s) dV_n(s)$ and $Z(t) = \int_0^t X(s) dV(s)$.
Because $X_n$ $\cC$-converges, for $M>0$, there is a finite partition $0=t_0< t_1< t_2 \cdots <t_m=M$ of $[0,M]$ such that
$\delta_{n,m} = |X_n-X_{n,m}|_{\infty,M}$  and $\delta_{\infty,m} = |X-X_{\infty,m}|_{\infty,M}$ are as small as one wants with large probability, where $X_{n,m}(t) = \sum_{k=1}^m X_n(t_k)\I_{t_{k-1}\le t < t_k}$ and $X_{\infty,m}(t) = \sum_{k=1}^m X(t_k)\I_{t_{k-1}\le t < t_k}$.
Explicitly, for every $M>0$ and $\epsilon>0$, there are integers $m=m(M,\epsilon)\ge1$ and $N=N(M,m,\epsilon)\ge1$ such that
$
P(\sup_{n\ge N}\delta_{n,m}\ge\epsilon)\le3P(\delta_{\infty,m}\ge\epsilon)<\epsilon$.
 Similarly $|Z_n-Z_{n,m}|_{\infty,M}\le \sqrt{V_n(M)}\delta_{n,m}$ and
 $|Z-Z_{\infty,m}|_{\infty,M}\le \sqrt{V(M)}\delta_{\infty,m}$ hold, where
$$
Z_{n,m}(t)=\int_0^t X_{n,m}(s) dV_n(s) = \sum_{k=1}^m X_n(t_k)\left\{V_n(t_k\wedge t)-V_n(t_{k-1}\wedge t)\right\},
$$
and $Z_{\infty,m}(t) = \int_0^t X_{\infty,m} (s)dV(s)$.

By Proposition \ref{prop:inversecontinuity} and Remark \ref{rem:M1increasing}, $(X_n,Y_n,V_n) \stackrel{f.d.d.}{\rightsquigarrow} (X,Y,V)$ holds,
hence so does $Z_{n,m} \stackrel{f.d.d.}{\rightsquigarrow} Z_{\infty,m}$ for each fixed $m\ge1$.
Since $|Z_n-Z_{n,m}|_{\infty,M}\stackrel{Pr}{\rightsquigarrow}0$ and
$|Z-Z_{\infty,m}|_{\infty,M}\stackrel{Pr}{\rightsquigarrow}0$ hold for each fixed $n\ge1$ and $M>0$, plus $V_n(M)$ is tight for every $M>0$,
$Z_{n,m} \stackrel{f.d.d.}{\rightsquigarrow} Z_{\infty,m}$ for each fixed $m\ge1$ yields $Z_n \stackrel{f.d.d.}{\rightsquigarrow} Z$.
All that is left to show is that $Z_n\circ Y_n \stackrel{f.d.d.}{\rightsquigarrow} Z\circ Y$. Since $Y_n(T)$ is tight, choosing $\epsilon>0$, one can find $M = M_\epsilon$ so that for all
$n\in\dN$, $P(|Y_n(T)|> M) \le \epsilon$ and $P(|Y(T)|> M) \le \epsilon$. Suppose in addition that
$0=t_0< t_1 < \cdots <t_m=M$ are points of continuity of the distribution of $Y(T)$. Then, on $\{Y_n(T)\le M\}$,
\begin{eqnarray*}
Z_{n,m}\circ Y_n(T) &=& \sum_{j=1}^m \I\{t_{j-1}< Y_n(T)\le t_j\}X_n(t_j)\{V_n\circ Y_n(T)-V_n(t_{j-1})\}\\
&&\quad  +\sum_{j=1}^m \I\{t_{j-1}< Y_n(T)\le t_j\}\left[\sum_{k=1}^{j-1} X_n(t_k)\{V_n(t_k)-V_n(t_{k-1})\} \right]\\
&=& \sum_{j=1}^m \I\{t_{j-1}< Y_n(T)\le t_j\}X_n(t_j)\{V_n\circ Y_n(T)-V_n(t_{j-1})\}\\
&&\quad  +\sum_{k=1}^{m-1} X_n(t_k)\{V_n(t_k)-V_n(t_{k-1})\} \I\{t_{k}< Y_n(T)\le M\}.
\end{eqnarray*}
Since $(X_n,Y_n,V_n)\stackrel{f.d.d.}{\rightsquigarrow}(X,Y,V)$, it follows that $Z_{n,m}\circ Y_n(T)$
converges in law to $Z_{\infty,m}\circ Y(T)$.
 The rest of the proof follows similarly by considering common points of continuity of the
 marginal distributions of $Y(T_1),Y(T_2),\ldots,Y(T_i)$
 for any choice of $0=T_0< T_1 < \cdots <T_i=T$, first yielding
 $Z_{n,m}\circ Y_n \stackrel{f.d.d.}{\rightsquigarrow} Z_{\infty,m}\circ Y$ for each fixed $m\ge1$ and then
 $Z_n\circ Y_n \stackrel{f.d.d.}{\rightsquigarrow} Z\circ Y$.
\qed

\subsection{Proof of Theorem \ref{thm:clt_mart}}\label{pf:clt_mart}
Hypothesis \ref{hyp:bm_limit} implies that $\crochet{ W_n }$ is $\mathcal{C}$-tight;
applying Theorem \ref{thm:jumps_vanish} to martingale $W_n = M_n\circ \tau_n$ yields
$(W_n, \crochet{ W_n } ) \stackrel{\mathcal{C}}{\rightsquigarrow} (W,\crochet{ W } )$,
with $W$ a standard Brownian motion since $\crochet{ W }_t=t$, a deterministic trajectory.
Hypothesis \ref{hyp:an_unbounded} implies $A_n \stackrel{\mathcal{M}_1}{\rightsquigarrow} A$
(in the light of Remark \ref{rem:M1increasing}), which yields $M_n \stackrel{\cL^2_w}{\rightsquigarrow} M$,
by Theorem \ref{thm:comp}.
The independence of $W$ and $A$ stems from Lemma \ref{lem:main} in the case where $A$ is an inhomogeneous L\'evy process.
In the case where Hypothesis \ref{hyp:disc_disc1} is valid instead, Remark \ref{rem:jumps_tau} yields
$(A_n,\tau_n) \stackrel{\mathcal{M}_1}{\rightsquigarrow} (A,\tau)$. Since the pair $(\tau,W)$ is
$\mathbb{F}_{\tau}$-adapted and $\tau$ is a non-decreasing inhomogeneous L\'evy process, it has no Brownian component;
therefore, it is independent of $W$ by Lemma \ref{lem:main}. Since $A(0)=0$ guarantees the continuity of the reverse mapping
$\tau\mapsto A$, $A$ is also independent of $W$ and so is the pair $(A,\tau)$ as the image of Borel measurable mapping
$\tau\mapsto(A,\tau)$.
The independence of $W$ and $A$ implies $(A_n,W_n) \stackrel{f.d.d.}{\rightsquigarrow} (A,W)$
from which there follows $(W_n,A_n,\tau_n)  \stackrel{f.d.d.}{\rightsquigarrow} (W,A,\tau)$
under either of the alternate conditions in the statement.
Going back to the end of the proof of Theorem \ref{thm:comp} which yielded
$\int_0^\cdot  M_n(s)ds \stackrel{f.d.d.}{\rightsquigarrow} \int_0^\cdot  M(s)ds$, note that
common points of continuity of the marginal distributions of $A(T_1),A(T_2),\ldots,A(T_i)$
also allow for the discretization of the integrals $\cI(A_n)$ and $\cI(W_n)$ along the same pattern as
for $\cI(M_n) $, since $D$-valued $A_n$ and $W_n$ have bounded trajectories on compact time sets.
One gets, through the same reasoning as for $\cI(M_n) $,
$\left( \cI(M_n) , \cI(A_n) ,  \cI(W_n) \right ) \stackrel{f.d.d.}{\rightsquigarrow}
\left( \cI(M) , \cI(A) ,  \cI(W) \right )$.
By Lemma \ref{lem:lem_pairs} and Remark \ref{rem:RdExtension2},
$(M_n,A_n,W_n) \stackrel{\cL^2_w}{\rightsquigarrow} (M,A,W)$ ensues.
\qed

\subsection{Proof of (\ref{cVisnablaF})}\label{pf:cVisnablaF}

We only do the case where $V(x) = v(|x|)$ is radial as an illustration. Then $F(x) = \pi^{-1}\int_{\dR^2}V(y) \log|y-x|dy$ is also radial, $\int_0^\infty(r v(r)dr = 0$ since $\bar V=0$, and
$\disp
F(x) = \Phi(|x|) = \frac{1}{2\pi}\int_0^\infty \int_0^{2\pi} r v(r) \log\left(|x|^2+r^2-2r|x|\cos\theta\right)d\theta dr$.
Setting $\gamma = 2ra/(r^2+a^2) = 2\sin\alpha\cos\alpha=\sin(2\alpha)$, then successively
$$
\frac{1}{2\pi}\int_0^{2\pi}  \log\left(1-\gamma\cos\theta \right)d\theta = \log\left(\frac{1+\sqrt{ 1-\gamma^2}}{2}\right) = \log\left\{\frac{\max\left(r^2,a^2\right)}{r^2+a^2}\right\},
$$
\begin{eqnarray*}
\Phi(a) &=& \int_0^\infty r v(r) \log\left(a^2+r^2\right)  dr
                  +  \frac{1}{2\pi}\int_0^\infty \int_0^{2\pi} r v(r) \log\left(1-\gamma\cos\theta \right)d\theta dr\\
&=& \int_0^\infty r v(r) \log\left(a^2+r^2\right)  dr + \int_0^\infty r v(r) \log\left(\frac{\max(r^2,a^2)}{r^2+a^2}\right)dr\\
&=& 2\int_0^\infty r v(r) \log{\max\left(r,a\right) }dr.
\end{eqnarray*}
As  a result,
$\disp
 F(x) = 2\int_0^\infty r v(r) \log{\max\left(r,|x|\right) }dr = 2 \int_{|x|}^\infty r v(r)\log\left(\frac{r}{|x|}\right) dr
 $
 and
$|\nabla F(x)|^2=\left\{\frac{2}{|x|}\int_{|x|}^\infty  r v(r) dr\right\}^2$, so that
\begin{eqnarray*}
\int |\nabla F(x)|^2 dx  &=& 8\pi \int_0^\infty \frac{1}{r}\left\{\int_r^\infty  sv(s)ds\right\}^2   dr\\
&=& -8\pi \int_0^\infty \int_0^\infty ra v(r)v(a) \log{\max(r,a)}drda\\
&=& -2 \int_0^\infty \int_0^\infty   \int_0^{2\pi}  ra v(r)v(a) \log{ \left(r^2+a^2-2ra \cos\theta \right) }d\theta drda
\end{eqnarray*}
which is precisely $c_V$.
\qed
\bibliographystyle{apalike}

\def\cprime{$'$}

\end{document}